\documentclass{amsart}
\usepackage{amssymb}

\newcommand{\ex}{\operatorname{ex}^{*}}
\newcommand{\exni}{\operatorname{ex}}
\newcommand{\floor}[1]{\left\lfloor#1\right\rfloor}

\newtheorem{theorem}{Theorem}[section]

\newtheorem{corollary}[theorem]{Corollary}
\newtheorem{question}[theorem]{Question}

\title{Large subposets with small dimension}
\author{Benjamin Reiniger* \and Elyse Yeager}
\thanks{*corresponding author; email: reinige1@illinois.edu\\
reinige1@illinois.edu, yeager2@illinois.edu\\Mathematics Dept., University of Illinois, Urbana-Champaign}   

\begin{document}
\begin{abstract}
Dorais asked for the maximum guaranteed size of a dimension $d$ subposet of an $n$-element poset.  A lower bound of order $\sqrt{n}$ was found by Goodwillie.  We provide a sublinear upper bound for each $d$.  For $d=2$, our bound is $n^{0.8295}$.
\end{abstract}

\maketitle

\section{Introduction}
Given a family of posets $\mathcal{F}$, let $\ex(P, \mathcal{F})$ denote the size of the largest induced subposet of $P$ that does not contain any member of $\mathcal{F}$ as an induced subposet.  
Similarly, $\exni(P,\mathcal{F})$ is the size of the largest induced subposet of $P$ that does not contain a member of $\mathcal{F}$ as a not-necessarily-induced subposet.  
These can be seen as poset analogues of the \emph{relative Tur\'an} numbers of families of graphs (in some host graph).
We write $\ex(P,\{Q\})$ as simply $\ex(P,Q)$.  
Let $\ex(n,\mathcal{F})$ denote the minimum of $\ex(P,\mathcal{F})$ over all $n$-element posets $P$.  
In other words, $\ex(n,\mathcal{F})$ is the maximum $k$ such that every $n$-element poset $P$ has an $\mathcal{F}$-free subposet of size at least $k$.  
Let $B_n$ be the boolean lattice of dimension $n$ and $A_n$ an antichain on $n$ points.

Then $\ex(P,B_1)$ is just the width of $P$ and $\ex(P,A_2)$ is the height of $P$.  
The function $\exni(B_n, B_2)$ is heavily studied as the maximum size of a ``diamond-free'' family of sets.  
In the literature, $\exni(B_n, P)$ is denoted $\operatorname{La}(n,P)$, and $\ex(B_n,P)$ is denoted $\operatorname{La}^{\sharp}(n,P)$ or $\operatorname{La}^*(n,P)$.

In this note we are concerned with finding large subposets of small dimension.  Hence we let $\mathcal{D}_d$ denote the family of posets of dimension at least $d$, and ask

\begin{question}
What is $\ex(n, \mathcal{D}_{d+1})$?
\end{question}

In other words, what is the largest size of a dimension $d$ subposet we are guaranteed to find in an $n$-element poset?  
(Note that when $d=1$, $A_n$ shows that $\ex(n, \mathcal{D}_{d+1})=1$.  We henceforth assume $d>1$.)  
This question was originally posed by F. Dorais~\cite{D-mo}, whose aim was to eventually understand the question for infinite posets~\cite{Dorais}.  
Goodwillie~\cite{Goodwillie} proved that $\ex(n, \mathcal{D}_{d+1})\geq\sqrt{dn}$ by considering the width of $P$: if $w(P)\geq\sqrt{dn}$, then a maximum antichain is a large subposet of dimension 2; if $w(P)\leq\sqrt{dn}$, then by Dilworth's theorem the union of some $d$ chains has $\geq \sqrt{dn}$ elements, and this has dimension at most $d$.

We provide a sublinear upper bound by considering the lexicographic power of standard examples.  Theorem \ref{thm:lexpower} finds the extremal number for lexicographic powers, and Corollary \ref{cor:2d} applies this to $\ex(n,\mathcal{D}_3)$.  For other $d$, Table \ref{table:dme} provides upper bounds on $\ex(n,\mathcal{D}_{d+1})$.

\section{Main theorem}
Given a poset $P$ and positive integer $k$, let $P^{k}$ denote the lexicographic order on $k$-tuples of elements of $P$.
\begin{theorem}
\label{thm:lexpower}
Let $P$ be a poset, $\mathcal{F}$ a family of posets, $k$ a positive integer, and let $n=|P|^k=|P^k|$.  Then $\ex(|P|^k, \mathcal{F}) \leq \ex(P^k, \mathcal{F})\leq n^{\log_{|P|}(\ex(P,\mathcal{F}))}$.
\end{theorem}
\begin{proof}
Let $S$ be a maximum $\mathcal{F}$-free subposet of $P^k$ (so $|S|=\ex(P^k,\mathcal{F})$).  For $i\leq k+1$ and each $i$-tuple $\alpha$, let 
\begin{align*}
 S_{\alpha} &= \{s\in S : \alpha\text{ is an initial segment of $s$}\}, \\
 Q(\alpha)&= \{p\in P: (\alpha, p) \text{ is an initial segment of some $s\in S$}\}.
\end{align*}
Then each $Q(\alpha)$ is an induced subposet of $S$, under any of the maps that assign to $p\in P$ an element $s\in S$ with initial segment $(\alpha,p)$.  Since $S$ is $\mathcal{F}$-free, so is $Q(\alpha)$, hence $|Q(\alpha)|\leq\ex(P,\mathcal{F})$.

We have that
\[ \left| S_{\alpha} \right| = \sum_{p\in Q(\alpha)} \left| S_{(\alpha,p)} \right|
 \leq |Q(\alpha)| \cdot \max_{p\in Q(\alpha)} \left| S_{(\alpha,p)} \right|
 \leq \ex(P,\mathcal{F}) \cdot \max_{p\in Q(\alpha)} \left| S_{(\alpha,p)} \right|. \]
When $\omega$ is a $k$-tuple, $S_{\omega}$ is either $\{\omega\}$ or $\emptyset$.  Hence we have, for $\alpha$ an $i$-tuple,
\[ \left| S_{\alpha} \right| \leq (\ex(P,\mathcal{F}) )^{k-i}, \]
and in particular, for $\alpha$ the 0-tuple,
\[ |S| \leq \left( \ex(P,\mathcal{F}) \right)^k = |P|^{\log_{|P|}(\ex(P,\mathcal{F})^k)} = n^{\log_{|P|}(\ex(P,\mathcal{F}))} . \]
\end{proof}

\begin{corollary}
\label{cor:2d}
For all sufficiently large $n$, 
$\ex(n, \mathcal{D}_3) \leq n^{0.8295}.$
\end{corollary}
\begin{proof}
Take $P=S_m$, the standard example on $2m$ points, in the preceding theorem.  It is easy to see that $\ex(S_m, \mathcal{D}_3)=m+2$.  Hence the exponent on the family of posets obtained is $\log_{2m}(m+2)$, which is minimized at $m=10$ with value approximately 0.82948.  This completes the proof when $n$ is a power of 20.

Otherwise, write $n=\sum_{i=0}^k \alpha_i (20)^i$, each $\alpha_i\in\{0,\dotsc,19\}$.  Then let $Q$ be the poset that is the disjoint union of $\alpha_i$ copies of $S_{10}^i$ for each $i$.  A maximum dimension 2 subposet of $Q$ is precisely the union of maximum dimension 2 subposets of each $S_{10}^i$.  So
\begin{align*}
 \ex(n, \mathcal{D}_3) &\leq \ex(Q,\mathcal{D}_3) \\
 &= \sum_{i=0}^k \alpha_i \ex(S_{10}^i, \mathcal{D}_3) \\
 &\leq \sum_{i=0}^k \alpha_i (20)^{0.82949i} \\
 &\leq \left(\sum_{i=0}^k \alpha_i \right) \left( \frac{\sum_{i=0}^k \alpha_i (20)^i }{\sum_{i=0}^k \alpha_i} \right)^{0.82949} & \text{(Jensen's inequality)} \\
 &= \left(\sum_{i=0}^k \alpha_i \right)^{1-0.82949} n^{0.82949} \\
 &\leq (19 (\floor{\log_{20} n} + 1) )^{0.17051} n^{0.82949} \\
 &< n^{0.8295}
\end{align*}
for sufficiently large $n$.
\end{proof}

Essentially the same proof works for any $d$.  We have for any $m$ and any $\epsilon>0$ that for sufficiently large $n$, $\ex(n,\mathcal{D}_{d+1})\leq n^{\log_{2m}(m+d)+\epsilon}$.  Table \ref{table:dme} shows some values of $d$ with the minimizing $m$ and the minimum value of the exponent (rounded to the 5th decimal place).
\begin{center}
\begin{table}[h!]
\begin{tabular}{rrl}
$d$ & $m$ & $\log_{2m}(m+d)$ \\
2 & 10 & 0.82948 \\
3 & 17 & 0.84953 \\
4 & 25 & 0.86076 \\
10&78 & 0.88663 \\
100&1169&0.92122
\end{tabular}
\caption{Values of $m$ that minimize $\log_{2m}(m+d)$ for given $d$.}
\label{table:dme}
\end{table}
\end{center}

\section{Remarks}
There is still a rather large gap between the known lower and upper bounds for $\ex(n,\mathcal{D}_{d+1})$.  Any improvement to either the lower or upper bound would be interesting.

Given the interest in $\exni(B_n, B_2)$, one may be interested in $\ex(B_n, \mathcal{D}_{d+1})$ instead of $\ex(n, \mathcal{D}_{d+1})$.  
\begin{question}
What is $\ex(B_n, \mathcal{D}_{d+1})$?
\end{question}
Lu and Milans (personal communication) 
have shown that $\ex(B_n, S_d)\leq (4d+C\sqrt{d}+\epsilon)\binom{n}{\floor{n/2}}$.  
Hence also $\ex(B_n, \mathcal{D}_d)=\Theta(\binom{n}{\floor{n/2}})$.  
For small cases, we have computed that $\ex(B_n, \mathcal{D}_3)=1,4,7,12,20$ for $n=1,2,3,4,5$.  

In 1974, Erd\H{o}s~\cite{Erdos} posed and partially answered the following question: given an $r$-uniform hypergraph $G_r(n)$ on $n$ vertices such that every $m$-vertex subgraph has chromatic number at most $k$, how large can the chromatic number of $G_r(n)$ be? 
Using probability methods Erd\H{o}s found a lower bound for ordinary graphs when $k=3$; that is, when every $m$-vertex subgraph has chromatic number at most 3.
Thinking of poset dimension as analogous to graph chromatic number, we ask:
\begin{question}
Given a poset $P$ with $n$ elements such that every $m$-element subposet has dimension at most $d$, how large can the dimension of $P$ be?
\end{question}

\section{Acknowledgements}
The authors would like to thank Stephen Hartke and Michael Ferrara for their mentorship and guidance. The authors acknowledge support from National Science Foundation grant DMS 08-38434 ``EMSW21-MCTP: Research Experience for Graduate Students''.

\end{document}